\documentclass[12pt,reqno]{amsart}
\usepackage{amssymb,amscd,amsmath}
\usepackage[alphabetic]{amsrefs}
\usepackage[all]{xy}
\usepackage[headings]{fullpage}
\usepackage[T1]{fontenc}
\usepackage{libertine}
\usepackage{stackengine}
\usepackage{mathtools}
\usepackage{graphicx}
\usepackage{multirow}
\usepackage{colortbl}
\usepackage{pdfsync}
\usepackage{tikz}

\numberwithin{equation}{section}

\theoremstyle{plain}

\newtheorem{prop}[subsection]{Proposition}

\newtheorem{lemma}[subsection]{Lemma}
\newtheorem{cor}[subsection]{Corollary}

\newtheorem*{thm*}{Theorem}
\newtheorem*{mthm*}{Main Theorem}
\newtheorem{prop-def}[subsection]{Proposition-Definition}
\newtheorem{def-prop}[subsection]{Definition-Proposition}
\newtheorem{def-lemma}[subsection]{Definition-Lemma}

\theoremstyle{definition}
\newtheorem{defn}[subsection]{Definition}

\newtheorem*{defn*}{Definition}

\theoremstyle{remark}
\newtheorem{rem}[subsection]{Remark}

\newtheorem{s-example}[subsection]{Example}
\newtheorem{s-examples}[subsection]{Examples}

\usepackage[OT2,T1]{fontenc}
\DeclareSymbolFont{cyrletters}{OT2}{wncyr}{m}{n}
\DeclareMathSymbol{\sha}{\mathalpha}{cyrletters}{"58}

\newcommand{\CC}{\mathcal{C}}

\newcommand{\GG}{{\mathcal{G}}}
\newcommand{\HH}{{\mathcal{H}}}

\newcommand{\EE}{\mathcal{E}}

\newcommand{\LL}{\mathcal{L}}
\newcommand{\MM}{\mathcal{M}}

\newcommand{\OO}{\mathcal{O}}

\newcommand{\Z}{\mathbb{Z}}

\newcommand{\R}{\mathbb{R}}
\newcommand{\C}{\mathbb{C}}

\newcommand{\tensor}{\otimes}

\newcommand{\nodiv}{\not|}

\def\nodiv{\mathrel{\mathchoice{\not|}{\not|}{\kern-.2em\not\kern.2em|}
{\kern-.2em\not\kern.2em|}}}

\newcommand{\SL}{\mathrm{SL}}

\newcommand{\pmat}[1]{\begin{pmatrix}#1\end{pmatrix}}

\DeclareMathOperator{\ord}{ord}

\DeclareMathOperator{\dvsr}{div}

\DeclareMathOperator{\spec}{Spec}

\newcommand{\RN}[1]{%
  \textup{\uppercase\expandafter{\romannumeral#1}}%
}

\def\clap#1{\hbox to 0pt{\hss#1\hss}}

\makeatletter
\newcommand*\bigcdot{\mathpalette\bigcdot@{.5}}
\newcommand*\bigcdot@[2]{\mathbin{\vcenter{
               \hbox{\scalebox{#2}{$\m@th#1\bullet$}}}}}
\makeatother

\numberwithin{equation}{section}
\DeclareMathOperator{\Ig}{Ig}
\renewcommand{\and}{\quad\text{and}\quad}

\usepackage{color} 
\begin{document}
\title{New unlikely intersections on elliptic surfaces}

\author{Douglas Ulmer}
\address{Department of Mathematics \\ University of Arizona
 \\ Tucson, AZ~~85721 USA}
\email{ulmer@arizona.edu}

\author{Jos\'e Felipe Voloch}
\address{School of Mathematics and Statistics, University of
  Canterbury, Private Bag 4800, Christchurch 8140, New Zealand} 
\email{felipe.voloch@canterbury.ac.nz}

\date{\today}


% descriptions:
%14J27(1980–now)Elliptic surfaces, elliptic or Calabi-Yau fibrations
%11G05(1980–now)Elliptic curves over global fields

%11G25(1980–now)Varieties over finite and local fields
%14D10(1973–now)Arithmetic ground fields (finite, local, global) and
%families or fibrations
%14G17(2010–now)Positive characteristic ground fields in algebraic geometry


\begin{abstract}
  Consider a Jacobian elliptic surface $\EE\to\CC$ with a section $P$
  of infinite order.  Previous work of the first author and Urz\'ua
  over the complex numbers gives a bound on the number of tangencies
  between $P$ and a torsion section of $\EE$ (an ``unlikely
  intersection''), and more precisely, an exact formula for the
  weighted number of tangencies between $P$ and elements of the
  ``Betti foliation''.  This work used analytic techniques that
  apparently do not generalize to positive characteristic.  In this paper,
  we extend their work to characteristic $p$, and we develop a
  second approach to tangency properties of algebraic curves on a
  complex elliptic surface, yielding a new family of unlikely
  intersections with a strong connection to a famous homomorphism of
  Manin.  We also correct inaccuracies in the literature about this
  homomorphism.
\end{abstract}

\maketitle

\section{Introduction}

\subsection{Motivation}
We study unlikely intersections in the form of tangencies of curves on
elliptic surfaces.  In \cite{UlmerUrzua21}, the first author and
Urz\'ua worked over the complex numbers and gave a bound on the number
of tangencies between a section of infinite order and $n$-torsion
(multi-)sections for varying $n$.  Such tangencies may be viewed as
``unlikely intersections.''  More generally, they introduced a
``Betti'' foliation (considered independently by several other
authors) of a dense open subset of the elliptic surface by
multisections generalizing the $n$-torsion multisections, and they
showed that there is an exact formula for the weighted number of
tangencies between the given section of infinite order and the leaves
of the Betti foliation.  These leaves are generally of infinite degree
over the base, so non-algebraic, and the bound on tangencies uses a
certain non-holomorphic (but real-analytic) 1-form.  It is thus not
evident whether or how these bounds might be generalized to base
fields of positive characteristic.

In this paper, we extend their work to positive characteristic. Our
results are somewhat similar but there are some differences as
well. In doing so, we noticed that there is a second way of looking at
tangency properties of algebraic curves on a complex elliptic surface,
which we develop. Our results in characteristic $p$ are somehow
intermediate between the results of \cite{UlmerUrzua21} and this
second approach to tangencies in characteristic zero.

For our purposes, an elliptic surface is a smooth projective algebraic
surface $\EE/k$, where $k$ is an algebraically closed field, equipped
with an elliptic fibration $\pi: \EE\to\CC$, where $\CC/k$ is a smooth
projective algebraic curve, admitting a zero section $O$. The generic
fiber of $\pi$ is thus an elliptic curve $E/K$, where $K=k(\CC)$ is
the function field of $\CC$.  We take a point $P\in E(K)$ of infinite
order and also denote by $P$ the corresponding section of $\pi$.

Our work depends crucially on a homomorphism of groups
$\mu: E(K) \to K$ which was introduced in characteristic zero by Manin
\cite{Manin63b} and generalized to positive characteristic by the
second author \cite{Voloch90}.  (This map was also studied
cohomologically in \cite{Ulmer91} and made very explicit in
\cite{Broumas97}.)  The value $\mu(P)$ is expressed as a rational
function of the coordinates of $P$ and their derivatives. Our results
on tangencies with $P$ are obtained by a careful study of the local
and global properties of $\mu$. We sharpen the local analysis of the
values of $\mu$ done by previous authors and relate them with the
tangency properties were are interested in.  Interpreting the values
of $\mu$ as sections of an appropriate bundle on $\CC$ leads to a
bound on the number of tangencies.

To give the flavor of our results, we state one result in
characteristic $p$ and one over $\C$.

First, let $k$ be algebraically closed of characteristic $p>3$, let
$N>3$ be an integer prime to $p$, and let $\CC=\Ig_1(N)$ be the
(Igusa) curve parameterizing elliptic curves with a point of order $N$
and an Igusa structure of order $p$.  Let $K=k(\CC)$, let $E/K$ be the
universal curve, and let $P\in E(K)$ be a point of infinite order.  We
define a local intersection invariant $I(P,t)$ for $t\in\CC$ which
measures the order of contact between $P$ and local sections which are
divisible by $p$.  (See Definition~\ref{def:I} for details.)  We
consider the set $T$ (``tangencies'') of points $t\in \CC$ where
$I(P,t)>1$, and we divide it into subsets $T_o$ and $T_s$ where the
elliptic curve corresponding to $t$ is ordinary or supersingular
respectively.  Our result is then a bound
\begin{align*}
      |T|\le |T_o|+p|T_s|&\le p(2g-2-d)+(p-1)\delta,\\
      &=p(p-1)\sigma-\delta,
\end{align*}
where $g$ is the genus of $\CC$, $d$ is the degree of a certain bundle
$\omega$, $\delta$ is the number of points where $E$ has bad
reduction, and $\sigma$ is the number of points where $E$ has
supersingular reduction.  In general, the division of $T$ into two
parts is dictated by the zeroes of a certain differential $\lambda$
which measures the failure of Kodaira-Spencer to be an isomorphism.

Now assume that $k=\C$.  We define ``complex Betti leaves'' which are
parameterized by two complex numbers modulo 1 and are such that on a
dense open subset of $\EE$, there is a complex Betti leaf through a
given point with a given tangent direction.  Thus, an unlikely
intersection is when $P$ meets some complex Betti leaf to order at
least 3.  We define an exceptional set $S\subset\CC$ (essentially
where $\EE$ has bad reduction or the period map is not an immersion)
and a set of tangencies $T_\C\subset\CC\setminus S$ where $P$ has
order of contact at least 3 with a complex Betti leaf.  We then find
an upper bound:
\[|T_\C|\le 4g-4-d+|S|.\]

In both characteristic zero and $p$, our arguments have a local aspect
(relating intersection numbers to orders of vanishing of a Manin map)
and a global aspect (interpreting the image of the Manin map as a
section of a suitable line bundle and deducing a bound from degree
considerations).  In both the complex and characteristic $p$ cases,
our results are somewhat looser than those of \cite{UlmerUrzua21} due
to subtleties in the local analysis.

\subsection{Acknowledgements}
The first author thanks L'Institut des Hautes \'Etudes Scientifiques
and La Fondation Sciences Math\'ematiques de Paris for support and
hospitality during a sabbatical year, the Simons Foundation for
support via Collaboration Grant 713699, and the participants in the
IHES seminar ``Équations diff\'erentielles motiviques et au–del\`a'' for
questions which stimulated this work.  The second author thanks the
Marsden Fund administered by the Royal Society of New Zealand for
financial support.

\section{$p$-descent and tangencies}\label{s:p-descent}
Let $k$ be a field of characteristic $p$ (algebraically closed and
with $p>3$ for convenience).  Let $\CC$ be a smooth, irreducible,
projective curve over $k$, and let $K=k(\CC)$. Let $E/K$ be an
elliptic curve with $j$-invariant $j\not\in K^p$,
% so the notation matches equation {eq:omega_q}
and let $P\in E(K)$ be a point of infinite order.

Let $\EE\to\CC$ be the N\'eron model of $E/K$ (as a group scheme) with
zero section $O$.  Let $\omega$ be the line bundle
$O^*\left(\Omega^1_{\EE/\CC}\right)$.  It is well known that if
\begin{equation}\label{eq:model}
 y^2=x^3+a_4x+a_6\qquad a_4,a_6\in K 
\end{equation}
is a short Weierstrass model of $E$, then $dx/2y$ defines a global
section of $\omega$, and that $a_4(dx/2y)^4$ and $a_6(dx/2y)^6$ are
well-defined global sections of $\omega^4$ and $\omega^6$ respectively
which are independent of the choice of short Weierstrass model.  If
$(x(P),y(P))$ are the coordinates of a point $P\in E(K)$ in the model
\eqref{eq:model}, then $x(P)(dx/2y)^2$ and $y(P)(dx/2y)^3$ are global
sections of $\omega^2$ and $\omega^3$ respectively, again independent
of the model.  More generally, a homogeneous polynomial of weight
$\kappa$ in $x(P),y(P),a_4,a_6$ (with weights 2, 3, 4, and 6
respectively) gives rise to a well-defined global section of
$\omega^\kappa$.  In order to compute the order of vanishing of such a
section at $t\in\CC$, we choose the model \eqref{eq:model} with $a_4$
and $a_6$ regular of minimal valuation at $t$ and then take the
valuation of the element of $K$ obtained by evaluating the polynomial
using these coordinates.

Define polynomials $M$, $A$, and $L$ in $a_4$, $a_6$, and $x$ by
\[\left(x^3+a_4x+a_6\right)^{(p-1)/2}=x^pM(x)+Ax^{p-1}+L(x)\]
where $L(x)$ consists of terms of degrees $<p-1$ in $x$.  Then $A$,
the Hasse invariant, is homogeneous of degree $p-1$ in $a_4$ and $a_6$
(and is independent of $x$), and $M(x)$ is homogeneous of degree
$p-3$.  Thus $A$ and $M$ give rise to sections of $\omega^{p-1}$ and
$\omega^{p-3}$ respectively.

There is a twisted differential on $\CC$ associated to $E$ which is
mentioned as $dt/t$ in \cite[\S3]{Voloch90}, considered more
explicitly and denoted $dq/q$ in \cite[\S5]{Ulmer91}, and calculated
in terms of a Weierstrass model and denoted $\omega_q$ in
\cite[\S4]{Broumas97}.  To avoid conflicting with other notation here,
we call it $\lambda$.  The formula of \cite{Broumas97} for it is:
\begin{equation}\label{eq:omega_q}
\lambda=\frac{a_4}{18a_6}\frac{dj}{j}.
\end{equation}
Since $\lambda$ is a 1-form times a homogeneous expression in $a_4$ and
$a_6$ of weight $-2$, $\lambda(dx/2y)^{-2}$ is a well-defined rational
section of $\Omega^1_\CC\tensor\omega^{-2}$ independent of the choice
of model.

\begin{lemma}\label{lemma:tau}
  Let $t$ be a point of  $\CC$ and let
  \[\ell=\ord_t\left(\lambda(dx/2y)^{-2}\right)\]
  where $\ord_t$ is as a section of
  $\Omega^1_\CC\tensor\omega^{-2}$.  Then we have
\begin{center}
\renewcommand{\arraystretch}{1.8}
\begin{tabular}{| c | c | c | c | c | c | c | c | c | c | c | c |}
\hline
  \text{Reduction at $t$}
  &$\RN{1}_0$&$\substack{\displaystyle\RN{1}_m\\(p| m)}$
  &$\substack{\displaystyle\RN{1}_m\\(p\nodiv m)}$
  &$\RN{2}$&$\RN{3}$&$\RN{4}$
  &$\substack{\displaystyle\RN{1}^*_m\\(p| m)}$
  &$\substack{\displaystyle\RN{1}^*_m\\(p\nodiv m)}$
  &$\RN{4}^*$&$\RN{3}^*$&$\RN{2}^*$\\
  \hline
  $\ell$
  &$\ge0$&$\ge0$&$=-1$
  &$\ge-1$&$\ge-1$&$\ge-1$
  &$\ge-1$&$=-2$
  &$\ge-2$&$\ge-2$&$\ge-2$\\
  \hline
\end{tabular}
\end{center}
\end{lemma}

\begin{proof}
  This can be deduced from the expression \eqref{eq:omega_q} using a
  table (such as those in \cite[p.~540]{MirandaPersson86} or
  \cite[p.~41]{MirandaBTES}) relating the valuations of $a_4$, $a_6$,
  and $dj$ to reduction types.  Alternatively, one may use the
  calculation of the divisor of $dq/q$ in \cite[\S7]{Ulmer91} on the
  Igusa curve and the diagram in Remark~\ref{rem:triangle} below.
\end{proof}

\begin{rem}
  If $\lambda$ is calculated with respect to a model of $E$ with $A=1$,
  the resulting differential is logarithmic.  This is the inspiration
  for the notations $dq/q$ and $dt/t$.  We do not use them here since in
  general $\lambda$ is not logarithmic.
\end{rem}

There is an injective homomorphism $\mu:E(K)/pE(K)\to K$ defined in
\cite{Voloch90} and computed explicitly in terms of a model in
\cite{Broumas97}, where the main theorem says:
\begin{equation}\label{eq:Broumas}
  \mu(P)=y(P)M(x(P))+
  \wp_A\left(\frac{dx(P)/2y(P)}{\lambda}
    -\frac{12x(P)^2+\left(\frac{d\Delta/\Delta}{\lambda}\right)x(P)
      +8 a_4}{12y(P)}\right).
\end{equation}
Here $\wp_A(z)=z^p-Az$ and $\Delta$ is the discriminant of
\eqref{eq:model}.

As is visible from the explicit formula, $\mu(P)$ is a homogeneous
rational function of degree $p$, so it gives rise to a rational
section of $\omega^p$.  Let $\nu(P):=\mu(P)\lambda(dx/2y)^{p-2}$.
This is a rational section of $\Omega^1_{\CC}\tensor\omega^{p-2}$
which is independent of the model.

Let $D_0$ and $D_\infty$ be the divisors of zeroes and poles of
$\lambda(dx/2y)^{-2}$ respectively, and define $D'$ to be the sum of
the points $t$ where $E$ has $I_m$ reduction with $m>0$ and the
order of $P$ in the component group at $t$ is divisible by $p$. (This
implies that $p|m$.)  Let
\[D=(p-1)D_0+pD'.\]
For a positive integer $n$, let $(nP.O)_t$ be the local intersection
number between $nP$ and the zero section $O$ over $t\in\CC$. 

\begin{prop}\label{prop:bound1}
  Suppose that $E$ has everywhere semi-stable reduction.  Then
  \[\nu(P)\in H^0(\CC,\Omega^1_\CC\tensor\omega^{p-2}(D)).\]
  Moreover, if for some positive integer $n$ prime to $p$ we have
  $\iota:=(nP.O)_t>1$, then
  \[\ord_t(\nu(P))\ge
    % \min\left(p(\iota-1)-\ord_t(D),\iota-1+\ord_t(A)\right)
    -\ord_t(D)+
    \begin{cases}
      p&\text{if $t|D$}\\
      1&\text{if $t\nodiv D$}
    \end{cases}
  \]
  \textup{(}Here $\ord$ is as a section of
  $\Omega^1_\CC\tensor\omega^{p-2}$\textup{)}.  Moreover, if $E$ has
  multiplicative reduction at $t$ and $t|D$, then the same inequality
  holds already when $(nP.O)_t>0$.
\end{prop}

\begin{proof}
  We work with one closed point $t\in\CC$ at a time and choose a model
  \eqref{eq:model} with $a_4$ and $a_6$ regular and minimal at $t$.

  First consider the case where $E$ has good reduction at $t\in \CC$
  and $P(t)$ is not the point at infinity.  Lemma~\ref{lemma:tau}
  shows that $\lambda$ is regular at $t$.  Also, $y(P)$ vanishes at $t$
  if and only if $P(t)$ is a point of order 2.  In this case, we
  replace $P$ by $2P$ and deal with this case just below.  (Since
  $p\neq2$, $\nu(P)$ and $\nu(2P)$ have the same valuation at $t$.)
  Thus we may assume that $x(P)$ and $y(P)$ are finite at $t$ and
  $y(P)$ does not vanish there.  It is then visible from the formula
  \eqref{eq:Broumas} that $\nu(P)$ is regular unless $\lambda$ vanishes,
  and more generally
  \[\ord_t(\nu(P))\ge-(p-1)\ord_t(\lambda)=-\ord_t(D).\]

  Next consider the case where $E$ has good reduction at $t$ and
  $\iota:=(nP.O)_t>0$ for some positive $n$ prime to $p$.  We may
  replace $P$ with $nP$ without affecting the valuation of $\nu(P)$.
  Doing so, $x(P)$ and $y(P)$ have poles at $t$ of orders $2\iota$ and
  $3\iota$ respectively.
  It follows from the definition of $A$, $M$, and $L$ that
  \begin{align*}
    \ord_t&\left(y(P)M(x(P))-\wp_A\left(\frac{x(P)^2}{y(P)}
            +\frac{2a_4}{3y(P)}\right)\right)\\
    &=\ord_t\left(\wp_A\left(\frac{a_4x(P)+3a_6}{3x(P)y(P)}\right)
            -\frac{y(P)L(x(P))}{x(P)^p}\right)\\
    &\ge\iota.    
  \end{align*}
Moreover,
  $\ord_t(dx(P)/2y(P))\ge\iota-1$, so we have
  \[\ord_t\left(\wp_A\left(\frac{dx(P)/2y(P)}{\lambda}\right)\lambda\right)
    \ge -\ord_t(D)+
    \begin{cases}
      p&\text{if $t|D$}\\
      1&\text{if $t\nodiv D$}.
    \end{cases}\]
Also,
\[\ord_t\left(\wp_A\left(\frac{(\frac{d\Delta/\Delta}{\lambda})x(P)}
      {12y(P)}\right)\lambda\right)
      \ge -\ord_t(D)+
    \begin{cases}
      p&\text{if $t|D$}\\
      1&\text{if $t\nodiv D$}.
    \end{cases}\]
Combining the displayed lower bounds yields the
  proposition at places where $E$ has good reduction.

 Now consider the case where $E$ has multiplicative reduction at $t$,
 say of type $I_m$.  We use the Tate parameterization of $E$ over the
 completion $K_t$, so that $E(K_t)\cong K_t^\times/q^\Z$ where $q\in
 K_t^\times$ has valuation $m$.  In this case $A=1$ and
 \cite[\S6]{Voloch90} shows that $\lambda=dq/q$ and
 \[\mu(P)=\wp\left(\frac{du/u}{\lambda}\right),\]
 where $P$ is the class of $u\in K_t^\times$.  If $p\nmid m$, we may
 replace $P$ by $mP$ and conclude that $u$ is a unit, so
 $\ord_t(du/u)\ge0$. It follows that
 \[\ord_t(\nu(P))\ge
   \begin{cases}
     -p-(p-1)\ord_t(\lambda)&\text{if $p|m$ (so $\ord_t(\lambda)\ge0$),}\\
     0&\text{if $p\nmid m$ (which implies $\ord_t(\lambda)=-1$).}
   \end{cases}\]
If $\iota=(nP.O)_t>1$, then 
\[\ord_t(\nu(P))\ge\min\left(p(\iota-1)-(p-1)\ord_t(\lambda),
    \iota-1\right).\]
If $p|m$ and $(nP.O)_t>0$ for some $n$ prime to $p$, then $du/u$ is
regular and
\[\ord_t(\nu(P))\ge -(p-1)\ord_t(\lambda)=-\ord_t(D)+p.\]
This completes the proof of the proposition.
\end{proof}

\begin{rem}\label{rem:bounds}
  The careful reader will note that the inequalities of
  Proposition~\ref{prop:bound1} can be made more precise as a function
  of $\iota$.  Indeed, one finds that if for some positive integer $n$
  prime to $p$ we have $\iota:=(nP.O)_t>0$, then
  \[\ord_t(\nu(P))\ge
    \min\left(p(\iota-1)-\ord_t(D),\iota-1+\ord_t(A)\right)\]
  Also, a similar analysis, but with many more
  cases, can be carried out at places of additive reduction.  We
  have not tried to optimize here because we will prove a stronger
  result in the next section.
  \end{rem}

We can use this result to bound tangencies between $nP$ and the zero
section: 
\begin{cor}\label{cor:bound1}
  Suppose that $E$ has everywhere semi-stable reduction and that for
  every point $t\in\CC$, $p$ does not divide the order of the
  component group at $t$.  Let
  \[T=\left\{t\in\CC\left.\mid (nP.O)_t>1\text{ for some $n$ prime to
          $p$}\right.\right\},\]
    \[T_s=\left\{t\in T\left.\mid \ord_t(\lambda)>0\right.\right\}\]
    and $T_o=T\setminus T_s$.  Then
    \[|T|\le |T_o|+p|T_s|\le p(2g-2-d)+(p-1)\delta,\] %
    where $g$ is the genus of $\CC$, $d=\deg(\omega)$, and $\delta$ is
    the number of points where $E$ has bad reduction.
  \end{cor}

  \begin{proof}
    Under the hypotheses, Lemma~\ref{lemma:tau} yields that $\lambda$ has
    $\delta$ simple poles, so its divisor of zeroes has degree
    $2g-2-2d+\delta$ and 
\[\deg(D)=(p-1)(2g-2-2d+\delta).\]
    The degree of the divisor of zeroes of $\nu(P)$ as a section of
    $\LL:=\Omega^1_\CC\tensor\omega^{p-2}(D)$ is $\deg(\LL)$ and one
    computes that
    \[\deg(\LL)=2g-2+(p-2)d+\deg(D)=p(2g-2-d)+(p-1)\delta.\]
    On the other hand, if $\iota=(nP.O)_t>1$, then using
    Proposition~\ref{prop:bound1} we find that as a section of $\LL$,
    $\ord_t(\nu(P))\ge1$ if $t\in T_o$ and $\ord_t(\nu(P))\ge p$ if
    $t\in T_s$.  This yields the stated inequality.
  \end{proof}

  \begin{s-example} (Igusa curves) Choose an integer $N>3$ and prime
    to $p$ and let $\CC=\Ig_1(N)$ be the Igusa curve with level $N$
    structure.  (Since $p>3$, $\Ig_1(N)$ is a fine moduli space whose
    points parameterize isomorphism classes of generalized elliptic
    curves $X$ together with a point of order $N$ on $X$ and a point
    of order $p$ on $X^{(p)}$.)  Let $\pi:\EE\to\Ig_1(N)$ be the
    universal curve, $K=k(\Ig_1(N))$, and $E/K$ the generic fiber of
    $\pi$.  Then $E$ has everywhere semistable reduction and the
    orders of its component groups divide $N$ so
    Corollary~\ref{cor:bound1} applies.  Using results from
    \cite{Ulmer91}, we can write the upper bound there in another
    form: Let $\sigma$ be the number of supersingular points on
    $\Ig_1(N)$.  Then we have
    \[d=\deg(\omega)=\sigma
      \and
      2g-2=p\sigma-\delta,\]
    and the divisor of $\lambda$ is
    \[\dvsr(\lambda)=(p-2)(ss)-(cusps),\]
    where $(ss)$ is the divisor of supersingular points and $(cusps)$
    is the divisor of cusps.
    Thus the upper bound in Corollary~\ref{cor:bound1} is
\[p(2g-2-d)+(p-1)\delta= p(p-1)\sigma-\delta.\]
    Moreover, the set $T$ of tangencies decomposes into
   \[T_s=\left\{t\in T\left.\mid\text{ $t$
           supersingular}\right.\right\}
     \and
    T_o=\left\{t\in T\left.\mid\text{ $t$
           ordinary or cuspidal}\right.\right\}.
     \]
 (Here we say $t$ is supersingular (resp. ordinary or cuspidal) if the
 elliptic curve associated to $t$ has the corresponding properties.)
 Corollary~\ref{cor:bound1} then says that
 \[|T_o|+p|T_s|\le p(p-1)\sigma-\delta.\]
  \end{s-example}

  \begin{s-example}
    % see sage file "tangencies" for details
    Taking $p=3$ and $N=5$ in the previous example (which is
    technically not included but nevertheless relevant), the curve
    $\CC=Ig_1(5)$ in characteristic 3 is rational, and one checks that
    the universal curve $E$ has $\sigma=2$ points of supersingular
    reduction, $\delta=8$ points of multiplicative reduction, and has
    good reduction elsewhere.  One also checks that $E(K)$ has positive
    rank, and for a certain non-torsion point $P$, a version of
    Broumas's formula valid in characteristic 3 allows one to compute
    that $\nu(P)$ has double poles at the 2 supersingular points, and
    so at most 4 zeroes.  Moreover, one finds that $\nu(P)$ has simple
    zeroes at four distinct, ordinary points, and it turns out that
    suitable multiples of $P$ are tangent to the zero section over
    each of these points.  We are therefore in a situation where the
    upper bound of Corollary~\ref{cor:bound1} is tight.
  \end{s-example}

  \begin{rem}\label{rem:triangle}
    One may relate the general case to the Igusa case as follows:
    Choose $N>3$ and prime to $p$.  Suppose $K=k(\CC)$.
    Given $E/K$, let $\tilde K=k(\tilde\CC)$ be an extension of $K$ where $E$
    acquires a point of order $N$ and an Igusa structure.  This data
    induces a morphism $\tilde\CC\to\Ig_1(N)$ and we have a diagram
    \begin{equation}\label{eq:triangle}
    \xymatrix{
    &\ar[dl]_{\pi_1}\ar[dr]^{\pi_2}\tilde\CC&\\
    \CC&&\Ig_1(N)}
\end{equation}
where the pull back of the universal curve over $\Ig_1(N)$ to
$\tilde\CC$ is isomorphic to the pull back of $\EE$ to $\tilde\CC$. 
One may choose $\tilde\CC$ so that $\pi_1$ is at worst tamely
ramified, and since the universal curve over $\Ig_1(N)$ is everywhere
semistable, $\pi_2$ has the property that the pull back of the bundle
$\omega$ on $\Ig_1(N)$ to $\tilde\CC$ is the bundle $\omega$ on
$\tilde\CC$ given by the recipe at the beginning of this section.
\end{rem}

\begin{rem}
    We discuss another approach to Proposition~\ref{prop:bound1}
    when $E$ has additive reduction.  In the following table, for a
    given reduction type at $t$, we record $e$, the ramification index
    at $t$ of an extension $\pi:\tilde\CC\to\CC$ over which $E$
    acquires semi-stable reduction, and $f$, the change in the bundle
    $\omega$ (i.e., the integer such that
    $\pi^*\omega_{\CC}\cong\omega_{\tilde\CC}(f\cdot t)$ near $t$).

  \begin{center}
\renewcommand{\arraystretch}{1.8}
\begin{tabular}{| c | c | c | c | c | c | c | c |}
\hline
  Reduction&$\RN{2}$&$\RN{3}$&$\RN{4}$&$\RN{1}_m^*$
  &$\RN{4}^*$&$\RN{3}^*$&$\RN{2}^*$\\
  \hline
  $e$&6&4&3&2&3&4&6\\
  \hline
  $f$&1&1&1&1&2&3&5\\
   \hline
\end{tabular}
%\caption{Local trivializations of $\omega^{-1}$}\label{table:1}
\end{center}
\medskip
Applying the bound of Remark~\ref{rem:bounds} to $\pi^*\nu(P)$ we
find
\begin{multline*}
  e\ord_t(\nu(P))+(e-1)-f(p-2)\\
  \ge\min\left( p(e\iota-1)-\ord_t(\tilde D),
    e\iota-1+e\ord_t(A)-f(p-1)\right),  
\end{multline*}
where
\begin{align*}
\ord_t(\tilde D)=
  &(p-1)(e\ord_t(\dvsr_0(\lambda))+e-1+2f)\\
  &+
  \begin{cases}
    ep&\text{if $E$ has $\RN{1}_m^*$ reduction with $m>0$ and $p|m$}\\
    0&\text{otherwise.}
  \end{cases}  
\end{align*}
A bit of algebra shows that this lower bound on $\ord_t(\nu(P))$ is
slightly worse than in the case of semi-stable reduction.
\end{rem}

\section{A more general class of tangencies}
\label{s:general}
One of the main lessons of \cite{UlmerUrzua21} is that considering a
more general class of tangencies (with general leaves of the Betti
foliation rather than just torsion leaves) leads to stronger results
for the weighted number of tangencies (even an exact formula in the
complex case).  In this section we consider an analog of the Betti
foliation and use it to again deduce a bound on the number of
tangencies.

To motivate the discussion, note that if $(nP.O)_t>0$ then $P(t)$ is
an $n$-torsion point, and $(nP.O)_t$ is also the intersection number
over $t$ of $P$ and the $n$-torsion multisection passing through
$P(t)$.  Since $n$-torsion multisections are $p$-multiples of other
$n$ torsion multisections, they are killed by $\nu$, so it is natural
that intersections with them give bounds on $\ord_t(\nu(P))$.

However, if $P(t)$ is not an $n$-torsion point for $n$ prime to $p$,
then there may be not be a global point or multisection $R$ such that
$pR(t)=P(t)$.  We remedy this by considering \emph{local} points $R$
with $pR(t)=P(t)$ and use them to bound $\ord_t(\nu(P))$.  This leads
to a significant strengthening of Corollary~\ref{cor:bound1}.

\subsection{}
In order to focus on the essential ideas, we work over the Igusa
curve in characteristic $>3$.
%It is probably possible to extend our methods to the case where
%$E$ has everywhere semi-stable reduction, but the general case maybe
%rather challenging.
For the rest of this section, let $k$ be an algebraically closed field
of characteristic $p>3$, fix an integer $N>3$ and prime to $p$, and
let $\CC=\Ig_1(N)$ be the Igusa curve parameterizing generalized
elliptic curves with Igusa structures of level $p$ and $\Gamma_1(N)$
structures. Let $K=k(\CC)$, let $\pi:\EE\to\CC$ be the N\'eron model
of the universal curve (as a group scheme, so $\pi$ is smooth but not
proper), and let $E/K$ be the generic fiber of $\pi$.  By definition,
there is a non-trivial point $Q_0\in E^{(p)}(K)$ of order $p$, and we
fix a choice of one such point.

Given a model of $E$ as at \eqref{eq:model}, let $x'$ and $y'$ be the
coordinates on $E^{(p)}$ satisfying 
    \[ y^{\prime\,2}=x^{\prime\,3}+a_4^px'+a_6^p,\] %
and let $\pi':\EE'\to\CC$ be the N\'eron model of $E^{(p)}$.  Let
$F:E\to E^{(p)}$ be the relative Frobenius, and let $V:E^{(p)}$ be
the Verschiebung, so that $VF=p_E$.  We use $F$ and $V$ also to
denote the corresponding morphisms between $\EE$ and $\EE'$.

\subsection{}
We recall another version of the formula of Broumas which is
well-adapted to the context of the Igusa curve.  As usual, let
$\omega=O^*(\Omega^1_{\EE/\CC})$ and recall that the Hasse invariant
$A$ naturally determines a section of $\omega^{p-1}$.  Since we are on
the Igusa curve, $A=\alpha^{p-1}$ where $\alpha(dx/2y)$ is a section
of $\omega$, and as explained in \cite[\S3.2]{Broumas97}, the choice
of $Q_0$ determines a specific choice of $\alpha$.  Let $\lambda$ be
the rational section of $\Omega^1_{\CC}\tensor\omega^{-2}$ discussed
in the previous section.

With these notations, the explicit formula of Broumas takes an
especially simple form over the Igusa curve:
\begin{equation}\label{eq:Broumas2}
\lambda=\alpha^{p-2}\frac{dx'(Q_0)}{2y'(Q_0)}
  \and
  \mu(P)=\alpha^p\wp\left(
    \frac{dx'(Q)/2y'(Q)}{dx'(Q_0)/2y'(Q_0)}\right)  
\end{equation}
where $Q$ is an algebraic point of $E^{(p)}$ with $VQ=P$ (see
\cite[Thm.~4.1 and \S4.2]{Broumas97}).  Note in particular that the
kernel of $\mu$ on $E(K)$ (resp.~$E(K_t)$) is $pE(K)$
(resp.~$pE(K_t)$).

The divisors of $\alpha$ and $\lambda$ were computed in
\cite[\S7]{Ulmer91}.  (Note that the calculations there used a model
of $E$ with $A=1$, whereas here we use a minimal regular model.)
We have:
\begin{equation}\label{eq:divs1}
\dvsr(\alpha)=(ss)\and\dvsr(\lambda)=(p-2)(ss)-(cusps),  
\end{equation}
so the formula of Broumas for $\lambda$ yields
\begin{equation}\label{eq:divs2}
  \dvsr\left(dx'(Q_0)/2y'(Q_0)\right)=-(cusps).
\end{equation}

Fix a point $P\in E(K)$ of infinite order and not in $pE(K)$, and
write $P$ as well for the corresponding section of $\pi$.  Recall that
$\nu(P)=\mu(P)\lambda$.  Let $\OO_t$ be the completed local ring of
$\CC$ at $t$.

\begin{prop}\label{prop:local-points-and-bounds}\mbox{}
  \begin{enumerate}
  \item   Let $z\in\EE$ be a closed point, and let $t=\pi(z)$.  Then
    there is a section $R$ of $\pi$ over $\spec\OO_t$ such that $pR(t)=z$.
  \item If $R$ is a local section of $\pi$ over $\spec\OO_t$ with
    $pR(t)=P(t)$ and if $\iota=(pR,P)_t$ is the local intersection
    number, then we have
\[\ord_t(\nu(P))\ge
  \begin{cases}
    \iota-1&\text{if $t$ is not a supersingular point,}\\
    p(\iota-1)-(p-1)(p-2)
       &\text{if $t$ is supersingular and $\iota< p$,}\\
    \iota+p-2
       &\text{if $t$ is supersingular and $\iota\ge p$.}
  \end{cases}\]
  \end{enumerate}
\end{prop}

\begin{proof}
  (1) Because $\EE$ and $\EE'$ have everywhere semistable reduction,
  the connected components of the identity
  $G=\left(\pi^{-1}(t)\right)^0$ and are either elliptic curves or
  multiplicative groups, and multiplication map $p:G\to G$ is
  surjective.  Moreover, since $\Ig_1(N)\to X_1(N)$ is completely
  split over the cusps, the component groups of $\EE$ have order
  dividing $N$ and thus of order prime to $p$.  Therefore, $p$ is also
  surjective on the component groups, and using the snake lemma, on
  all of $\EE$.  Suppose that $w\in \EE$ satisfies $pw=z$.  Since
  $\pi$ is smooth, by Hensel's lemma $w$ lifts to a point $R$ with
  values in $\OO_{t}$, and we have $pR(t)=pw=z=P(t)$, as desired. 

  (2) With $P$ and $R$ as in the statement, let $R_0=P-pR$, a local
  section of $\pi$ which specializes to the identity at $t$.  Since
  $\nu$ kills $pE(K_t)$, we have $\nu(P)=\nu(R_0)$.  Moreover,
  \[\iota:=(pR,P)_t=(O,R_0)_t,\] %
  so we are reduced to showing that if $R_0$ is a local section of
  $\pi$ which specializes to the identity and if $\iota=(O,R_0)_t$,
  then $\nu(R_0)$ is bounded below by the quantity displayed in the
  statement of part (2).

  If $t$ is not a supersingular point, then $V$ is \'etale over $t$
  and we may choose $Q_0\in E^{(p)}(\OO_t)$ with $VQ_0=R_0$ and with
  $Q_0$ specializing to the identity at $t$.  With this choice, $Q_0$
  must meet the identity section to order $\iota$, i.e., we have
  $\ord_t(x'(Q_0))=-2\iota$ and $\ord_t(y'(Q_0))=-3\iota$.  We find
  that
\[\ord_t(dx'(Q)/2y'(Q))\ge \iota-1,\]
and formulas \eqref{eq:Broumas2}, \eqref{eq:divs1}, and
\eqref{eq:divs2} then imply that
\[\ord_t(\nu(P))\ge \iota-1.\]

Now assume that $t$ is supersingular.  Then the points $Q_0$ with
$VQ_0=R_0$ are defined over an extension $\OO'$ of $\OO_t$ which may
be ramified, and in order to compute the different, we consider
$V:E^{(p)}\to E$ in the formal groups of $E$ and $E^{(p)}$.  It takes
the shape
\[ V(z)=Az+uz^p+\cdots\] %
where $u$ is a unit at $t$.  Since $R_0$ meets the zero section to
order $\iota$, the choice of $Q_0$ meeting the zero section
corresponds to a solution of
\[v\pi_t^\iota=Az+uz^p+\cdots,\] %
where $v$ is a unit.

Since $A$ vanishes to order $p-1$, a consideration of slopes shows
that if $\iota$ is $<p$, $\OO'$ is (wildly) ramified over $\OO$ with
ramification break $m=p-\iota$.  We find that $dx'(Q_0)/2y'(Q_0)$
vanishes in $\Omega^1_{\OO'/k}$ to order $(p-1)(p-\iota+1)$, and
$Q_0$ meets the zero section to order $\iota$.  Using
\eqref{eq:Broumas2}, \eqref{eq:divs1}, and \eqref{eq:divs2}, we find
that
\begin{align*}
  \ord_t(\nu(P))&\ge (2p-2)+(\iota-1)-(p-1)(p-\iota+1)\\
  &=p(\iota-1)-(p-1)(p-2).
\end{align*}

If $\iota\ge p$, then the choice of $Q_0$ meeting the zero section is
defined over $\OO_t$ and meets the zero section to order
$\iota-(p-1)$, and using \eqref{eq:Broumas2}, \eqref{eq:divs1}, and
\eqref{eq:divs2}, we find that
\[\ord_t(\nu(P))\ge \iota+p-2.\]

This completes the proof of the proposition.
\end{proof}

As a reality check, note that when $P(t)$ is torsion and $R$ is chosen
to be a torsion multisection, then the bounds in
Proposition~\ref{prop:local-points-and-bounds} agree with
those in Remark~\ref{rem:bounds}.  There is another reality
check coming from \cite[Thm.~5.5]{Ulmer91}, where $\mu(P)$ is related
to the ramification over $t$ of the extension $K_t(Q_0)/K_t$, in
agreement with our conductor calculation above.

Part (1) of the proposition shows that the $\sup$ in the following
definition is finite, so $I(P,t)$ is a well-defined integer.

\begin{defn}\label{def:I}
  For a point of infinite order $P\in E(K)\setminus pE(K)$ and a
  closed point $t\in\CC$, let
  \[I(P,t)=\sup_{R}\left((pR,P)_t\right),\]
  where $R$ runs through sections of $\pi$ over $\spec\OO_t$ with
  $pR(t)=P(t)$.
\end{defn}

When $P(t)$ is $n$ torsion, then $I(P,t)\ge(P,pR)_t=(nP,O)_t$ where
$pR$ is the $n$-torsion multisection through $P(t)$.  We do not know
of an example where the inequality is strict.

By definition, $I(P,t)\ge1$, and $I(P,t)\ge2$ if and only if $Q$ is
tangent over $t$ to a local point divisible by $p$.  These are the
``more general tangencies'' in the title of this section.  

The same argument establishing Corollary~\ref{cor:bound1} yields:

\begin{cor}
  Let
  \[T=\left\{t\in\CC\left| I(P,t)>1\right.\right\},\]
    \[T_s=\left\{t\in T\left| \text{$t$ is supersingular}\right.\right\}\]
    and $T_o=T\setminus T_s$.  Then
    \begin{align*}
      |T|\le |T_o|+p|T_s|&\le p(2g-2-d)+(p-1)\delta,\\
      &=p(p-1)\sigma-\delta,
    \end{align*}
    where $g$ is the genus of $\CC$, $d=\deg(\omega)$, $\delta$ is the
    number of points 
    where $E$ has bad reduction, and $\sigma$ is the number of points
    where $E$ has supersingular reduction.
\end{cor}

\section{Complex Betti leaves and the Manin map}
In this section, we work over the complex numbers.  It was already
noted in \cite[Thm.~2.8]{CorvajaDemeioMasserZannier22} that the Manin
map can be used to bound tangencies with the Betti foliation, although
the connection there is rather loose and does not directly give
finiteness.  Here, we define a new class of tangencies via
``complex Betti leaves,'' interpret them in terms of the Manin map,
and deduce a strong upper bound.

In order to compute examples, we need an explicit algebraic version of
the Manin map, and we note, as did the authors of
\cite{CorvajaDemeioMasserZannier22}, that published versions
of this map are incorrect, so we give a detailed correction.

Throughout this section, $\pi:\EE\to\CC$ is an elliptic
fibration with zero section $O$ and another section $P$ of infinite
order.  Except in Remark~\ref{rem:isotrivial}, we assume that $\EE$ is
non-isotrivial. 

\subsection{Complex Betti leaves}\label{ss:unif}
Our approach is closely related to that in \cite{UlmerUrzua21}, so we
use the notation of that paper freely here.  In particular, let
$\CC^0$ be the open over which $\pi$ is smooth, choose a base point
$b\in\CC^0$ and a basis of $H_1(\pi^{-1}(b),\Z)$,  let
$\widetilde{\CC^0}$ be the universal cover of $\CC^0$, and let
$\Gamma=\pi_1(\CC^0,b)$.  Then we have a
monodromy representation $\rho:\Gamma\to\SL_2(\Z)$ with coordinates
\[\rho(\gamma)=\pmat{a_\gamma&b_\gamma\\c_\gamma&d_\gamma},\]
a period map $\tau:\widetilde{\CC^0}\to\HH$ to the upper half-plane
$\HH$, a cocycle
\[f_\gamma(\tilde t)=\left(c_\gamma\tau(\tilde t)+d_\gamma\right)^{-1},\]
and a diagram 
\begin{equation}\label{eq:unif}
\xymatrix{\left(\widetilde{\CC^0}\times\C\right)/\Z^2\ar[r]\ar[d]
  &\left(\widetilde{\CC^0}\times\C\right)/(\Gamma\ltimes\Z^2)
         \ar^{\quad\qquad\sim}[r]\ar[d]
    &\EE^0\ar[d]\ar@{^{(}->}[r]&\EE\ar[d]^{\pi}\\
    \widetilde{\CC^0}\ar[r]&\widetilde{\CC^0}/\Gamma\ar^{\sim}[r]&
    \CC^0\ar@{^{(}->}[r]&\CC,}  
\end{equation}
of complex manifolds where $\Gamma\ltimes\Z^2$ acts on
$\widetilde{\CC^0}\times\C$ by
\[\left(\gamma,m,n\right)(\tilde t,w)
  =\left(\gamma\tilde t,f_{\gamma}(\tilde t)
    (w+m\tau(\tilde t)+n)\right).\]

Let $\GG$ be the direct image on $\CC$ of the $\Z^2$ local system on
$\CC^0$ associated to $\rho$, and let $\omega^{-1}$ be the inverse of
the line bundle $O^*(\Omega^1_{\EE/\CC})$.  This is an extension to
$\CC$ of the line bundle on $\CC^0$ associated to the cocycle
$f_\gamma$.  Let $\EE^{id}$ be the sheaf of sections of $\pi$ which
land in the identity component of each fiber.  Then we have an exact
sequence
\begin{equation}\label{eq:id-comp}
0\to\GG\to\omega^{-1}\to\EE^{id}\to0,  
\end{equation}
which is a sheaf-theoretic manifestation of the uniformization
\eqref{eq:unif} extended over all of $\CC$.

In \cite{UlmerUrzua21}, we defined Betti leaves by considering, for
fixed $(r,s)\in\R^2$, the image in $\EE^0$ of the set
\[\left\{(\tilde t,r\tau(\tilde t)+s)\left|\ \tilde
      t\in\widetilde{\CC^0}\right.\right\}
\subset \widetilde{\CC^0}\times\C,\]
This is a (typically non-closed) submanifold of $\EE^0$ which depends
only on the class of $(r,s)$ in $(\R/\Z)^2/\Gamma$, and these
submanifolds give a foliation of $\EE^0$ by (complex) curves.  We
refer to \cite[\S3]{UlmerUrzua21} for more details.

In this paper, we define \emph{complex Betti leaves} associated to
pairs $(a,b)\in\C^2$ where the ``leaf'' associated to $(a,b)$ is
the image of the set 
\[\left\{(\tilde t,a\tau(\tilde t)+b)\left|\ \tilde
      t\in\widetilde{\CC^0}\right.\right\} \subset
  \widetilde{\CC^0}\times\C,\]
in $\EE^0$.  This image is a (typically non-closed) submanifold of
$\EE^0$ which we denote by $\LL_{a,b}$ and which depends only on the
class of $(a,b)$ in $(\C/\Z)^2/\Gamma$.

The term leaf is somewhat abusive since these submanifolds do not form
a foliation of $\EE^0$: whereas a Betti leaf is determined by any one
of its points, over the open set where $\tau$ is an immersion there is
a complex Betti leaf passing through a given point and with a given
derivative.  (So lifts of complex Betti leaves foliate an open subset
of the projectivized tangent bundle.)  In particular, given a section
$P$ of $\pi$ and a point $t\in\CC^0$ where $\tau'(t)\neq0$, we can
find a complex Betti leaf which meets $P$ to order at least 2.  The
``unlikely intersections'' we will study occur when this intersection
multiplicity is at least 3.

Let $t\in\CC\setminus\CC^0$ be a point of bad reduction. We say that a
complex Betti leaf is \emph{invariant} at $t$ if it is fixed by the
local monodromy, and \emph{vanishing} if not.
Arguing as in \cite[\S3.4]{UlmerUrzua21} yields the following
information on the local behavior of complex Betti leaves.

\begin{lemma}\label{lemma:inv-leaves}\mbox{}
  \begin{enumerate}
  \item A complex Betti leaf $\LL_{a,b}$ extends to a section of $\pi$
    in a neighborhood of $t$ if and only if it is an invariant leaf.
    In this case, it meets the closed fiber over $t$ in a point at which
    $\pi$ is smooth.
  \item If $\EE$ has reduction of type $I_n$ at $t$
    \textup{(}$n>0$\textup{)}, then the invariant leaves $\LL_{a,b}$
    are those where $a\in\frac1n\Z$, and every smooth point of the
    fiber $\pi^{-1}(t)$ is the limit of an invariant leaf.
    \item If $\EE$ has additive reduction at $t$, then there are
        only finitely many invariant leaves at $t$, and their images
        in the special fiber are exactly the torsion points of the
        fiber.
     \item A section is a complex Betti leaf if and only if it is
       torsion. 
  \end{enumerate}
\end{lemma}

(In the additive case, the order of the torsion points in the special
fiber is the same as the order of the component group, namely 1, 2, 3,
or 4.)

\begin{rem}\label{rem:isotrivial}
  If $\EE\to\CC$ is isotrivial, then the period function $\tau$ is
  constant and the complex Betti leaves are just the usual Betti
  leaves.  They are all closed, and the form $\tilde\eta$ used in
  \cite{UlmerUrzua21} to detect tangencies is just $dw$.  In other
  words, in the isotrivial case, tangencies are detected by zeroes of
  $P\mapsto P^*(dx/2y)$, a degenerate form of Manin's map.
\end{rem}

\subsection{Intersection numbers}
We define a local intersection index $I(P,t)$ as the maximum order of
contact at $P(t)$ between $P$ and the complex Betti leaves through
$P(t)$:
\[I(P,t):=\max_{a,b}\left(P.\LL_{a,b}\right)_t.\]
(This is finite because of the last item in
Lemma~\ref{lemma:inv-leaves}.)   Over an non-empty open subset of
$\CC^0$ (made more precise below), we have $I(P,t)=2$, so the
exceptional set of tangencies is the set of $t$ where $I(P,t)\ge3$.

We now describe the index $I(P,t)$ in terms of the uniformization of
the last section.  For a holomorphic function $h(z)$ defined on a disk
around $z=0$, let $\mu(h)$ be the order to which $h$ takes its value
at 0:
\[\mu(h)=\ord_{z=0}\left(h-h(0)\right).\]
Clearly $I(nP,t)=I(P,t)$, so we may replace $P$ with a multiple and
assume that it passes through the identity component of $\EE$ in each
fiber, i.e., that $P$ is a section of $\EE^{id}$.  We may then lift
$P$ (in the exact sequence \eqref{eq:id-comp}) to a section of
$\omega^{-1}$ near $t$ of the form $h(z)w_{-1}$ where $w_{-1}$ is a
generating section and $h$ is holomorphic.  If $E$ has good reduction
at $t$,
\[I(P,t)=\max_{a,b\in\C}\ord_{z=0}\left(h-a\tau-b\right).\]
There are choices of $(a,b)$ which makes the $\ord$ at least 1 (for
example, $a=0$ and $b=h(0)$), and in the typical situation where
$\tau'(0)\neq0$, there is a unique choice of $(a,b)$ which makes the
$\ord$ at least 2.  In general, we have
\[I(P,t)=
  \begin{cases}
    \mu(h)&\text{if $\mu(h)\neq\mu(\tau)$}\\
    \ge\mu(h)+1&\text{if $\mu(h)=\mu(\tau)$}
  \end{cases}\]
Note that in this case $I(P,t)\ge1$  and $I(P,t)\ge2$ if $\tau'(t)\neq0$.

If $E$ has multiplicative reduction at $t$, then the invariant leaves
are parameterized by $b\in\C$, there is a unique leaf meeting $P$ (the
one corresponding to $b=h(0)$), and we have
\[I(P,t)=\ord_{z=0}\left(h-h(0)\right)=\mu(h)\ge1.\]

If $E$ has additive reduction at $t$, the only invariant leaf is the
zero section, and we have
\[I(P,t)=\ord_{z=0}\left(h(z)\right)\ge0.\]

\subsection{The Manin map}
Let $K=\C(\CC)$ and let $E/K$ be the generic fiber of $\pi:\EE\to\CC$.
Then the set of sections of $\pi$ can be identified with the
Mordell-Weil group $E(K)$.  In \cite{Manin63b}, Manin defines a
homomorphism
\[M:E(K)\to K\]
which depends on the choice of an invariant differential on $E$ and on
a ``Picard-Fuchs operator'' $L$, i.e., a second-order differential
operator which kills the periods of the fibers of $\pi$.  (We will
make this more explicit in one example below.)  In terms of the
uniformization of Section~\ref{ss:unif} and a Weierstrass equation
$y^2=f(x)$, we have
\begin{equation}\label{eq:M}
M(P)=\frac{d^2}{d\tau^2}\int_O^P\frac{dx}y  
\end{equation}
where the integral is a function on $\widetilde{\CC^0}$ given by
integrating from $O$ to $P(t)$ along a path in the fiber
$\pi^{-1}(t)$. The values of the integral are only well defined up to
periods, and these are killed by the second-order derivative, so
$M(P)$ defines a meromorphic function on $\CC$.  As explained in
\cite{Manin98}, one may make this independent of choices by dividing
by the symbol of $d^2/d\tau^2$ (i.e., multiplying by the
quadratic differential $(d\tau)^{\tensor2}$) and dividing by the
section $dx/y$ of $\omega$.  This yields a map
\begin{equation}\label{eq:MM}
%P\mapsto\MM(P):=\left(\frac{\partial^2}{\partial\tau^2}\int_O^P\frac{dx}y\right)
%  \frac{(d\tau)^2}{dx/y}  
P\mapsto\MM(P):=\left(\frac{d^2}{d\tau^2}\int_O^P\frac{dx}y\right)
  \frac{(d\tau)^2}{dx/y}  
\end{equation}
from $E(K)$ to
meromorphic sections of the line bundle
\[\left(\Omega^1_\CC\right)^{\tensor2}\tensor\omega^{-1}.\]
We will quantify the possible poles of $\MM(P)$ in
Proposition~\ref{prop:bounds} below. 

An alternative point of view is afforded by the exact
sequence \eqref{eq:id-comp}.  If $P$ lies in $\EE^{id}$, computing the
integral $\int_O^Pdx/y$ amounts to lifting $P$ to a section of
$\omega^{-1}$.  A simple calculation shows that taking the second
derivative with respect to $\tau$ yields a (rational) section of
$\omega^3$, and the Kodaira-Spencer map
$\omega\to\Omega^1_\CC\tensor\omega^{-1}$ then gives a section of
$\left(\Omega^1_\CC\right)^{\tensor2}\tensor\omega^{-1}$.

Define $J(P,t)$ as the order of $\MM(P)$:
\[J(P,t):=\ord_t\MM(P),\]
where the order is computed as a section of
$\left(\Omega^1_\CC\right)^{\tensor2}\tensor\omega^{-1}$. 

Define a set $S$ of points of $\CC$ by
\begin{multline*}\label{eq:1}
S=\left\{t\in\CC\mid\text{$\EE$ has bad reduction at $t$}\right\}
\cup \left\{t\in\CC\mid\text{$j(t)\neq0,1728$ and $\ord_t(dj)>0$}\right\}\\  
\cup\left\{t\in\CC\mid\text{$j(t)=0$ and $\ord_t(dj)>2$}\right\}
\cup \left\{t\in\CC\mid\text{$j(t)=1728$ and $\ord_t(dj)>1$}\right\}.
\end{multline*}

The points in $S$ are where either $\EE$ has bad reduction or a
branch of the period $\tau$ is not a uniformizer:

\begin{lemma}
  If $t\in\CC\setminus S$ and $\tilde t\in\CC^0$ lies over $t$, then
  $\tau'(\tilde t)\neq0$.
\end{lemma}

\begin{proof}
This follows from the well-known fact that $j$, considered as a map
from the upper half plane to $\C$, is unramified away from $j=0$ or
$1728$, and it is ramified of order 3 over $j=0$ and of order 2 over
$j=1728$.   
\end{proof}

The following is the key connection between intersection numbers and the
Manin map.

\begin{prop}\label{prop:bounds}
  If $t\not\in S$,
  \[J(P,t)=I(P,t)-2\ge0.\]
  If $t\in S$, 
  \[J(P,t)\ge-1.\]
%    \begin{cases}
%      0&\text{if $\EE$ has $I_0^*$ reduction at $t$ and $\ord_t(dj)\le...$,}\\
%      -1&\text{otherwise.}
%    \end{cases}
\end{prop}

\begin{proof}
Suppose that $t\not\in S$ and choose a Weierstrass equation $y^2=f(x)$
for $E$ which is regular minimal at $t$.  Since $\tau$ is a
uniformizer at $t$, $(d\tau)^2/(dx/y)$ is a generating section of
$\left(\Omega^1_\CC\right)^{\tensor2}\tensor\omega^{-1}$ near $t$.
Comparing \eqref{eq:MM} and \eqref{eq:M}, we see that
$J(P,t)=\ord_t\MM(P)$ is the usual ord at $t$ of the function $M(P)$:
\[J(P,t)=\ord_tM(P).\]
Expanding the integral $\int_O^Pdx/y$ as a series in $\tau$ near $t$,
say
\[\int_O^Pdx/y=a_0+a_1\tau+a_I\tau^I+\cdots\]
(where $a_0$ and $a_1$ might be zero), we find
\[\ord_t\left(\frac{d^2}{d\tau^2}\int_O^P\frac{dx}y\right)
  =I(P,T)-2\ge0.\]

Now suppose that $t\in S$, and choose a uniformizer $z$ at $t$.  We
rewrite $d^2/d\tau^2$ in terms of $z$:
\[\frac{d}{d\tau}=\left(\frac{d\tau}{dz}\right)^{-1}\frac{d}{dz}
  \and
  \frac{d^2}{d\tau^2}=\left(\frac{d\tau}{dz}\right)^{-2}\frac{d^2}{dz^2}
  -\left(\frac{d\tau}{dz}\right)^{-3}\frac{d^2\tau}{dz^2}\frac{d}{dz}.
\]  
Computing the right hand side of \eqref{eq:MM} divided by the
generating section $(dz)^2/(dx/y)$ of
$\left(\Omega^1_\CC\right)^{\tensor2}\tensor\omega^{-1}$ near $t$,
we obtain
\[H=\frac{d^2}{dz^2}\int_O^P\frac{dx}y
  -\left(\frac{d\tau}{dz}\right)^{-1}\frac{d^2\tau}{dz^2}
  \frac{d}{dz}\int_O^P\frac{dx}y,\]
so $J(P,t)=\ord_t(H)$.  The derivatives of the integrals are regular,
but the factor
\[d\log\left(\frac{d\tau}{dz}\right)
  =\left(\frac{d\tau}{dz}\right)^{-1}\frac{d^2\tau}{dz^2}\]
could have a simple pole (and does except when $\tau$ is a
uniformizer at $t$ and this might happen if $\EE$ has reduction
type $I_0^*$ there).  Thus
$J(P,t)\ge-1$ for $t\in S$.

This completes the proof of the proposition.
\end{proof}

\begin{rem}
  The proof shows why we cannot bound $I(P,t)$ above in terms of
  $J(P,t)$ at points $t\in S$:  If $0<m<n$,
  \[\int_O^P\frac{dx}y=a_0+a_mz^m+a_nz^n+O(z^{n+1}),\]
  and
  \[\tau=b_0+b_mz^m+O(z^{n+1}),\]
  then for generic choices of the coefficients, $I(P,t)=n$ whereas
  $J(P,t)=m-2$, so $J(P,t)<I(P,t)-(n-m+1)$.  
\end{rem}

\begin{cor}\label{cor:bound-C}
  Let
  \[T_\C=\left\{t\in\CC\mid \text{$t\not\in S$ and $I(P,t)>2$}\right\}.\]
  Then
  \[|T_\C|\le 4g-4-d+|S|.\]
\end{cor}

\begin{proof}
  Indeed,
  \[|T_\C|\le\sum_{t\not\in S} \left(I(P,t)-2\right)=\sum_{t\not\in
      S}J(P,t)\]
  and the latter is the sum of all but finitely many ords of a
  regular section of
  $\left(\Omega^1_\CC\right)^{\tensor2}\tensor\omega^{-1}(S)$. 
\end{proof}

\subsection{Corrected algebraic Manin map}
As above, let $K=\C(\CC)$ and let $E$ over $K$ be given by a
Weierstrass equation $y^2=f(x)$.  It is not difficult to express the
Manin map for this situation in algebraic coordinates, i.e., in terms
of the function field $K$.  However, there are at least two
incorrect statements of the formula in the literature, so we give a
careful derivation here which we will use below in examples.  

Choose a non-trivial $\C$-derivation $\delta$ of $K$ and let $\eta$ be
the meromorphic differential on $\C$ dual to $\delta$.  Extend
$\delta$ to $K$ by setting $\delta x=0$. Extend $\delta$ to
differentials of $E$ by setting $\delta(h\,dx)=(\delta h)\,dx$.  We use
a prime ($'$) to denote the action of $\delta$ on elements of $K$ to
distinguish this action from the lifts.

Suppose that $A,B,C\in K$ are such that
\[L=A\delta^2+B\delta+C\]
is a second order differential operator with
\[L\left(\frac{dx}y\right)=dF\]
with $F\in K(E)$, i.e., $L(dx/y)$ is exact.  Then we get a
well-defined Manin map
\[M(P):=L\int_0^P\frac{dx}y\]
and its invariant version
\[\MM(P)=M(P)\frac{\eta^2}{A(dx/y)}\]
which is a meromorphic section of
$\left(\Omega^1_\CC\right)^{\tensor2}\tensor\omega^{-1}$ 
independent of the choices of $\delta$ and the model of $E$.

The Manin map kills 2-torsion points (indeed, any torsion points), so
it suffices to give a formula for it which is valid when $y(P)\neq0$.

\begin{prop}\label{prop:Manin-coords}
  We have
  \[M(P)=F(P)-F(O)+
    A\left(\frac{x(P)'(-\delta f)(P)}{2y(P)^3}
      +\left(\frac{x(P)'}{y(P)}\right)'\right)
    +B\left(\frac{x(P)'}{y(P)}\right).\]
\end{prop}

\begin{proof}
The tricky point is to evaluate the terms coming from the fact
that the limits of integration vary on $\CC$.  The model $y^2=f(x)$
presents $E$ as a double cover of the Riemann sphere.  Making branch
cuts (say from 0 to 1 and from $t$ to $\infty$ and choosing a branch
of $y=f^{1/2}$ we have
\[\int_O^P\frac{dx}y=\int_{\infty}^{x(P)}\frac1y\,dx
=\int_{\infty}^{x(P)}f^{-1/2}\,dx\]
  where the second integral is along a path in the Riemann sphere
  avoiding the branch cuts.  We have
  \[\delta\int_{\infty}^{x(P)}\frac1y\,dx
    =\int_{\infty}^{x(P)}\delta\left(\frac1y\right)\,dx
    +\frac{x(P)'}{y(P)},\]
  and since $\delta(1/y)=(-\delta y)/y^2=(-\delta f)/(2y^3)$, 
  \[\delta^2\int_{\infty}^{x(P)}\frac1y\,dx
    =\int_{\infty}^{x(P)}\delta^2\left(\frac1y\right)\,dx
    +\frac{x(P)'(-\delta f)(P)}{2y(P)^3}
    +\left(\frac{x(P)'}{y(P)}\right)'.\]
  (The second term on the right is the source of trouble in other
  versions.  The correct order is to apply the $\delta$ to the
  integrand first and then evaluate.)
  Putting everything together,
  \begin{align*}
    L\int_O^P\frac{dx}y&=\int_O^PL\frac{dx}y
                         +A\left(\frac{x(P)'(-\delta f)(P)}{2y(P)^3}
                         +\left(\frac{x(P)'}{y(P)}\right)'\right)
                         +B\left(\frac{x(P)'}{y(P)}\right)\\
    &=F(P)-F(O)                     +A\left(\frac{x(P)'(-\delta f)(P)}{2y(P)^3}
                         +\left(\frac{x(P)'}{y(P)}\right)'\right)
                         +B\left(\frac{x(P)'}{y(P)}\right). 
  \end{align*}
This completes the proof of the proposition.
\end{proof}

\subsection{Examples}
\subsubsection{Manin for Legendre}
  Let $K=\C(t)$ with $\delta=d/dt$, and let $E$ be given by
  \[E:\quad y^2=x(x-1)(x-t).\]
Then it is well known (and easy to check by hand) that
  \[\left(t(1-t)\delta^2+(1-2t)\delta-\frac14\right)\left(\frac{dx}y\right)
    =d\left(\frac{y}{2(x-t)^2}\right).\]
  Applying the proposition, we find that
  \[M(P)=\frac{y(P)}{2(x(P)-t)^2}
    +t(1-t)\left(\frac{x(P)'x(P)(x(P)-1)}{2y(P)^3}+\left(\frac{x(P)'}{y(P)}\right)'\right)
    +(1-2t)\left(\frac{x(P)'}{y(P)}\right).\]

\subsubsection{Tangencies on Legendre 1}
  Now let $L/K$ be the quadratic extension $L=\C(s)$ where
  $t=2-s^2/2$ and consider the point $P=(2,s)$ of $E(L)$.  
  Applying the formula above, we find that
  \[\MM(P)=\frac{-8}{s(s^2-4)(s^2-2)}\frac{(ds)^2}{(dx/y)},\]
  with divisor
  \[\dvsr(\MM(P))=2(\infty)-(s(s^2-4)(s^2-2)).\]
  Since the support of the divisor of $\MM(P)$ is contained in the
  exceptional set $S$, Corollary~\ref{cor:bound-C} shows that there
  are no higher order tangencies between $P$ and the complex Betti
  foliation away from $S$, i.e.,
  \[T_\C=\emptyset.\]
  
\subsubsection{Tangencies on Legendre 2}
  Similar results apply for the point $(a,s)$ where
  $a\in\C\setminus\{0,1\}$ and $s^2=a(a-1)(a-t)$.  With $P=(a,s)$, 
we find that
\[\MM(P)=\frac{-2a^2(a-1)^2}{s\left(s^2-a^2(a-1)\right)\left(s^2-a(a-1)^2\right)}
  \frac{(ds)^2}{(dx/y)},\]
  with divisor
  \[\dvsr(\MM(P))=2(\infty)-(s(s^2-a^2(a-1))(s^2-a(a-1)^2)).\]
  Again, the support of the divisor of $\MM(P)$ is contained in the
  exceptional set $S$, so Corollary~\ref{cor:bound-C} shows that there
  are no higher order tangencies between $P$ and the complex Betti
  foliation away from $S$, i.e.,
  \[T_\C=\emptyset.\]

\subsubsection{Tangencies on Legendre 3}
Now let $L$ be the bi-quadratic extension of $K=\C(t)$ generated by
$s_2$ and $s_3$ with $s_2^2=2(2-1)(2-t)$ and $s_3^2=3(3-1)(3-t)$ so that
$P_2=(2,s_2)$ and $P_3=(3,s_3)$ are points in $E(L)$.  Let
$Q=3P_3-P_2$.  One finds that $L$ is
again a rational function field $\C(u)$ where $s_2=(u^2-6u+3)/(u^2-3)$
and $s_3=(-3u^2 + 6u - 9)/(u^2 - 3)$ and that
$\MM(Q)$ vanishes to order 1 at $u=1$ (where $s_2=1$, $s_3=3$,
and $t=3/2$).  Inspecting the discriminant and $j$-invariant of $E/L$
shows that $u=1$ is not in the exceptional set $S$, so
Proposition~\ref{prop:bounds} shows that $Q$ meets a complex Betti
leaf to order $3$ over $u=1$.  Thus, while unlikely, such higher
intersections do indeed occur.

\bibliography{database}

\end{document}